\documentclass[11pt]{amsart}
\usepackage{color,amssymb}
\usepackage{bbm}
\usepackage[normalem]{ulem}

\usepackage[all]{xy}

\newtheorem{theorem}{Theorem}[section]
\newtheorem{claim}[theorem]{Claim}
\newtheorem{subclaim}[theorem]{Subclaim}

\newtheorem{corollary}[theorem]{Corollary}
\newtheorem{lemma}[theorem]{Lemma}
\newtheorem{problem}[theorem]{Problem}

\theoremstyle{definition}
\newtheorem{remark}[theorem]{Remark}
\newtheorem{definition}[theorem]{Definition}
\newtheorem{example}[theorem]{Example}

\newcommand{\IN}{\mathbb N}

\newcommand{\C}{\mathcal C}

\newcommand{\w}{\omega}

\newcommand{\Ra}{\Rightarrow}
\newcommand{\Haus}{\mathsf{T_{\!2}S}}

\newcommand{\Zero}{\mathsf{T_{\!z}S}}

\newcommand{\korin}[2]{\!\sqrt[#1]{\!#2}}

\usepackage{relsize}
\newcommand*{\defeq}{\stackrel{\mathsmaller{\mathsf{def}}}{=}}

\title{Categorically closed unipotent semigroups}
\author{Taras Banakh and Myroslava Vovk}
\address{T.Banakh: Ivan Franko National University of Lviv (Ukraine) and Jan Kochanowski University in Kielce (Poland)}
\email{t.o.banakh@gmail.com}
\address{M.Vovk: Department of Advanced Mathematics, Lviv Polytechnic National University, Lviv 79013,Ukraine}
\email{mira.i.kopych@gmail.com}

\makeatletter
\@namedef{subjclassname@2020}{%
  \textup{2020} Mathematics Subject Classification}
\makeatother


\begin{document}
\begin{abstract} Let $\C$ be a class of $T_1$ topological semigroups, containing all Hausdorff zero-dimensional topological semigroups. A semigroup $X$ is 
{\em $\C$-closed} if $X$ is closed in any topological semigroup $Y\in\C$ that contains $X$ as a discrete subsemigroup; $X$ is {\em injectively $\C$-closed} if for any (injective) homomorphism $h:X\to Y$ to a topological semigroup $Y\in\C$, the image $h[X]$ is closed in $Y$. A semigroup $X$ is {\em unipotent} if it contains a unique idempotent. We prove that a unipotent commutative semigroup $X$ is (injectively) $\C$-closed if and only if $X$ is bounded, nonsingular (and group-finite). This characterization implies that for every injectively $\C$-closed unipotent semigroup $X$, the center $Z(X)$ is injectively $\C$-closed. 
\end{abstract}
\maketitle

\section{Introduction and Main Results}

In many cases,  completeness properties of various objects of General Topology or  Topological Algebra can be characterized externally as closedness in ambient objects. For example, a metric space $X$ is complete if and only if $X$ is closed in any metric space containing $X$ as a subspace. A uniform space $X$ is complete if and only if $X$ is closed in any uniform space containing $X$ as a uniform subspace. A topological group $G$ is Ra\u\i kov complete  if and only if it is closed in any topological group containing $G$ as a subgroup.

On the other hand, for topological semigroups there are no reasonable notions of (inner) completeness. Nonetheless we can define many completeness properties of semigroups via their closedness in ambient topological semigroups.

A {\em topological semigroup} is a topological space $X$ endowed
with a continuous associative binary operation $X\times X\to
X$, $(x,y)\mapsto xy$.

\begin{definition} Let $\C$ be a class of topological semigroups.
A topological
semigroup $X$ is called
\begin{itemize}
\item {\em $\C$-closed} if for any isomorphic topological
embedding $h:X\to Y$ to a topological semigroup $Y\in\C$
the image $h[X]$ is closed in $Y$;
\item {\em injectively $\C$-closed} if for any injective continuous homomorphism $h:X\to Y$ to a topological semigroup $Y\in\C$ the image $h[X]$ is closed in $Y$;
\item {\em absolutely $\C$-closed} if for any continuous homomorphism $h:X\to Y$ to a topological semigroup $Y\in\C$ the image $h[X]$ is closed in $Y$.
\end{itemize}
\end{definition}

For any topological semigroup we have the implications:
$$\mbox{absolutely $\C$-closed $\Ra$ injectively $\C$-closed $\Ra$ $\C$-closed}.$$

\begin{definition} A semigroup $X$ is defined to be ({\em injectively, absolutely}) {\em $\C$-closed\/} if so is $X$ endowed with the discrete topology.
\end{definition}

In this paper we will be interested in the (absolute, injective) $\C$-closedness for the classes:
\begin{itemize}
\item $\mathsf{T_{\!1}S}$ of topological semigroups satisfying the separation axiom $T_1$;
\item $\Haus$ of Hausdorff topological semigroups;
\item $\Zero$ of Hausdorff zero-dimensional topological
semigroups.
\end{itemize}
A topological space satisfies the separation axiom $T_1$ if all its finite subsets are closed.
A topological space is {\em zero-dimensional} if it has a base of
the topology consisting of {\em clopen} (=~closed-and-open) sets.

Since $\Zero\subseteq\Haus\subseteq\mathsf{T_{\!1}S}$, for every semigroup we have the implications:
$$
\xymatrix{
\mbox{absolutely $\mathsf{T_{\!1}S}$-closed}\ar@{=>}[r]\ar@{=>}[d]&\mbox{absolutely $\mathsf{T_{\!2}S}$-closed}\ar@{=>}[r]\ar@{=>}[d]&\mbox{absolutely $\mathsf{T_{\!z}S}$-closed}\ar@{=>}[d]\\
\mbox{injectively $\mathsf{T_{\!1}S}$-closed}\ar@{=>}[r]\ar@{=>}[d]&\mbox{injectively $\mathsf{T_{\!2}S}$-closed}\ar@{=>}[r]\ar@{=>}[d]&\mbox{injectively $\mathsf{T_{\!z}S}$-closed}\ar@{=>}[d]\\
\mbox{$\mathsf{T_{\!1}S}$-closed}\ar@{=>}[r]&\mbox{$\mathsf{T_{\!2}S}$-closed}\ar@{=>}[r]&\mbox{$\mathsf{T_{\!z}S}$-closed.}
}
$$
From now on we assume that $\C$ is a class of topological semigroups such that $$\mathsf{T_{\!z}S}\subseteq\C\subseteq\mathsf{T_{\!1}S}.$$
Semigroups having one of the above closedness properties are called categorically closed. Categorically closed topological groups were investigated in
~\cite{AC,AC1,Ban,DU,G,Z1,KOO,L,Z2}. 
This paper is a continuation of the papers \cite{Ban}, \cite{BBm}, \cite{CCCS}, \cite{BB2}, \cite{ACS}  containing inner characterizations of semigroups possessing various categorically closed properties.

In this paper we shall characterize (absolutely and injectively) $\C$-closed {\em unipotent} semigroups. 

A semigroup $X$ is called
\begin{itemize}
\item {\em unipotent} if $X$ contains a unique idempotent;
\item {\em chain-finite} if any infinite set $I\subseteq X$ contains elements $x,y\in I$ such that $xy\notin\{x,y\}$;
\item {\em singular} if there exists an infinite set $A\subseteq X$ such that $AA$ is a singleton;
\item {\em periodic} if for every $x\in X$ there exists $n\in\IN$ such that $x^n$ is an idempotent;
\item {\em bounded} if there exists $n\in\IN$ such that for every $x\in X$ the $n$-th power $x^n$ is an idempotent;
\item {\em group-finite} if every subgroup of $X$ is finite;
\item {\em group-bounded} if every subgroup of $X$ is bounded.
\end{itemize}

The following characterization of $\C$-closed commutative semigroups was proved in the paper \cite{CCCS}.

\begin{theorem}[Banakh--Bardyla]\label{t:C-closed} A commutative semigroup is $\C$-closed if and only if it is chain-finite, periodic, nonsingular and group-bounded.
\end{theorem}

For unipotent semigroups, this characterization can be simplified as follows.

\begin{theorem}\label{t:main1} A unipotent semigroup $X$ is $\C$-closed if and only if $X$ is bounded and nonsingular.
\end{theorem}

Another principal result of this paper is the following characterization of injectively $\C$-closed unipotent semigroups.

\begin{theorem}\label{t:main2} A unipotent commutative semigroup $X$ is injectively $\C$-closed if and only if $X$ is bounded, nonsingular and group-finite.
\end{theorem}

\begin{example} For an infinite cardinal $\kappa$, the {\em Taimanov semigroup} $T_\kappa$ is the set $\kappa$ endowed with the semigroup operation $$xy=\begin{cases}1,&\mbox{if $x\ne y$ and $x,y\in\kappa\setminus\{0,1\}$};\\
0,&\mbox{otherwise}.
\end{cases}
$$The semigroup $T_\kappa$ was introduced by Taimanov in \cite{Taimanov}. Its algebraic and topological properties were investigated by Gutik \cite{Gutik} who proved that the semigroup $T_\kappa$ is injectively $\mathsf{T_{\!1}S}$-closed. The same fact follows also from  Theorem~\ref{t:main2} because the semigroup $T_\kappa$ is unipotent, bounded, nonsingular and group-finite. The Taimanov semigroups witness that there exist injectively $\mathsf{T_{\!1}S}$-closed unipotent semigroups of arbitrarily high cardinality.
\end{example}

For a semigorup $X$ let 
$$Z(X)\defeq\{z\in X:\forall x\in X\;\;(xz=zx)\}$$be the {\em center} of $X$.
The center of an (injectively) $\C$-closed semigroup has the following properties, proved in Lemmas~5.1, 5.3, 5.4 of \cite{CCCS} (and Theorem~1.7 of \cite{GCCS}).

\begin{theorem}[Banakh--Bardyla]\label{t:center} The center $Z(X)$ of any (injectively) $\C$-closed semigroup is chain-finite, periodic, nonsingular (and group-finite).
\end{theorem}

\begin{corollary}\label{c:iC} The center $Z(X)$ of an injectively $\C$-closed unipotent semigroup $X$ is injectively $\C$-closed.
\end{corollary}

\begin{proof} By Theorem~\ref{t:center}, the semigroup $Z(X)$ is chain-finite, periodic, nonsingular, and group-finite. By Theorem~\ref{t:C-closed}, the semigroup $Z(X)$ is $\C$-closed. By Theorem~\ref{t:main1}, $Z(X)$ is bouned. If $Z(X)$ is empty, then $Z(X)$ is injectively $\C$-closed. So, we assume that $Z(X)\ne\emptyset$. Being bounded, the semigroup $Z(X)$ contains an idempotent. Being a subsemigroup of the unipotent semigroup $X$, the semigroup $Z(X)$ is unipotent.  By Theorem~\ref{t:main2}, the unipotent bounded nonsingular group-finite semigroup $Z(X)$ is injectively $\C$-closed.
\end{proof}

Another corollary of Theorem~\ref{t:main2} describes the center of an absolutely $\C$-closed unipotent semigroup.

\begin{corollary}\label{c:aC} The center $Z(X)$ of an absolutely $\C$-closed unipotent semigroup $X$ is finite and hence absolutely $\C$-closed.
\end{corollary}

\begin{proof} By Theorem~\ref{t:center}, the semigroup $Z(X)$ is chain-finite, periodic, nonsingular, and group-finite. If $Z(X)$ is empty, then $Z(X)$ is finite and hence absolutely $\C$-closed. So, we assume that $Z(X)$ is not empty. Being periodic, the semigroup $Z(X)$ contains an idempotent $e$. Since $X$ is unipotent, $e$ is a unique idempotent of the semigroups $X$ and $Z(X)$. Let $H_e$ be the maximal subgroup of the semigroup $Z(X)$. The group $H_e$ is finite because $Z(X)$ is group-finite.  By Theorem~1.7 of \cite{GCCS}, the complement $Z(X)\setminus H_e$ is finite and hence the set $Z(X)$ is finite, too.
\end{proof}

\begin{remark} Theorem~\ref{t:main2} is an essential ingredient in the characterizations of injectively $\C$-closed commutative semigroups, obtained in \cite{ICT1S} and \cite{ICCS}.
\end{remark}

Corollaries~\ref{c:iC} and \ref{c:aC} suggest the following open problems.

\begin{problem}\label{prob}
\begin{enumerate}
\item Is the center of a $\C$-closed semigroup $\C$-closed? 
\item Is the center of an injectively $\C$-closed semigroup injectively $\C$-closed?
\item Is the center of an absolutely $\C$-closed semigroup absolutely $\C$-closed?
\end{enumerate}
\end{problem}

\begin{remark} By \cite{ICT1S}, for the class $\C=\mathsf{T_{\!1}S}$ of $T_1$ topological semigroups, the answer to Problem~\ref{prob}(2,3) is affirmative. Also some partial answers to Problem~\ref{prob}(3) for classes $\mathsf{T_{\!z}S}\subseteq\C\subseteq\mathsf{T_{\!2}S}$ are given in \cite{ACS}.
\end{remark}

\section{Preliminaries}\label{s:prelim}

We denote by $\w$ the set of finite ordinals and by $\IN\defeq\w\setminus\{0\}$ the set of positive integer numbers.

For an element $a$ of a semigroup $X$ the set
$$H_a\defeq\{x\in X:(xX^1=aX^1)\;\wedge\;(X^1x=X^1a)\}$$
is called the {\em $\mathcal H$-class} of $a$.
Here $X^1\defeq X\cup\{1\}$ where $1$ is an element such that $1x=x=x1$ for all $x\in X^1$.

By Corollary 2.2.6 \cite{Howie}, for every idempotent $e$ of a semigroup $X$ its $\mathcal H$-class $H_e$ coincides with the maximal subgroup of $X$, containing the idempotent $e$.

For a subset $A$ of a semigroup $X$ and a positive integer number $n$, let
$$\korin{n}{A}=\{x\in S:x^n\in A\}\quad\mbox{and}\quad\korin{\infty}{A}=\bigcup_{n\in\IN}\korin{n}{A}=\{x\in S:A\cap x^\IN\ne\emptyset\},$$ where $$x^\IN=\{x^n:n\in\IN\}$$is the {\em monogenic semigroup} generated by $x$. 

The following lemma is proved in \cite[3.1]{CCCS}.

\begin{lemma}\label{l:C-ideal} For any idempotent $e$ of a semigroup  we have $(\!\korin{\infty}{H_e}\cdot H_{e})\cup(H_{e}\cdot \korin{\infty}{H_e}\,)\subseteq H_{e}.$
\end{lemma}

\section{Proof of Theorem~\ref{t:main1}}

Theorem~\ref{t:main1} will be derived from the following lemmas.

\begin{lemma}\label{l:bounded-e} Let $X$ be a periodic commutative semigroup with a unique idempotent $e$ and trivial maximal subgroup $H_e$. If $X$ is not bounded, then there exists an infinite subset $A\subseteq X$ such that $AA=\{e\}$.
\end{lemma}

\begin{proof} To derive a contradiction, assume that $X$ is not bounded but for every infinite set $A\subseteq X$ we have $AA\ne\{e\}$. Taking into account that $X$ is periodic and unipotent, we conclude that $X=\korin{\infty}{H_e}$. By Lemma \ref{l:C-ideal}, the maximal subgroup $H_e=\{e\}$ is an ideal in $X$. 

Inductively we shall construct  a sequence of points $(x_k)_{k\in\w}$ and a  sequence of positive integer numbers $(n_k)_{k\in\w}$ such that for every $k\in\w$ the following conditions are satisfied:
\begin{itemize}
\item[(i)] $x_k^{n_k}\notin\{e\}\cup\{x_i^{n_i}:i<k\}$;
\item[(ii)] $x_k^{2n_k}=e$;
\item[(iii)] $\max_{i<k}|x_i^{n_i}X|<n_k$.
\end{itemize}
To start the inductive construction, take any $x_0\in X\setminus\{e\}$ and let $n_0$ be the smallest number such that $x^{n_0+1}_0=e$. Such number $n_0$ exists as $X$ is periodic. Since $\{e\}=H_e$ is an ideal in $X$, it follows from $x_0^{n_0+1}=e$ and $2n_0\ge n_0+1$ that $x_0^{2n_0}=e$. Assume that for some $k\in\IN$, we have chosen sequences $(x_i)_{i<k}$ and $(n_i)_{i<k}$. For every $i<k$, consider the set $x_i^{n_i}X$ and observe that for every $a,b\in X$ the inductive condition (ii) implies $$(x_i^{n_i}a)(x^{n_i}_ib)=x_i^{2n_i}ab=eab=e.$$ This means that $(x_i^{n_i}X)^2=\{e\}$ and by our assumption, the set $x_i^{n_i}X$ is finite. Since $X$ is unbounded, there exists an element $x_k\in X$ and a number $m_k>k+\max_{i<k}|x_i^{n_i}X|$ such that $x_k^{k+m_k}\ne e$ but $x_k^{1+k+m_k}=e$. Since the set $\{x_k^{j}:m_k\le j\le m_k+k\}$ consists of $k+1$ points, there exist a number $n_k\in\{m_k,\dots,m_k+k\}$ such that $x_k^{n_k}\notin\{x_i^{n_i}:i<k\}$. It follows from $x_k^{k+m_k}\ne e=x_k^{1+k+m_k}$ and $2n_k\ge m_k+m_k\ge 1+k+m_k>n_k$ that $x_k^{n_k}\ne e=x_k^{2n_k}$. This completes the inductive step.
\smallskip

After completing the inductive construction, consider the infinite set $A=\{x_k^{n_k}:k\in\w\}$. We claim that $x_i^{n_i}x_k^{n_k}=e$ for any $i\le k$. For $i=k$ this follows from the inductive condition (ii). So, assume that $i<k$. By the induction condition (iii) and the Pigeonhole Principle,  there exist two positive numbers $j<j'\le |x_i^{n_i}X|+1\le n_k$ such that $x_i^{n_i}x_k^{j}=x_i^{n_i}x_k^{j'}$. Let $d=j'-j\le n_k$ and observe that $x_i^{n_i}x_k^{j}=x_i^{n_i}x_k^{j+d}$. Then $$x_i^{n_i}x_k^{j+2d}=x_i^{n_i}x_k^{j+d}x_k^{d}=x_i^{n_i}x_k^{j}x_k^{d}=x_i^{n_i}x_k^{j+d}=x_i^{n_i}x_k^{j}.$$Proceeding by induction, we can prove that $x_i^{n_i}x_k^{j}=x_i^{n_i}x_k^{j+pd}$ for every $p\in\w$. Since $X$ is periodic and $\{e\}=H_e$ is an ideal in $X$, there exists $p\in\IN$ such that $x_k^{j+pd}=e$ and hence $x_i^{n_i}x_k^{j}=x_i^{n_i}x_k^{j+pd}=e$.  Then $x_i^{n_i}x_k^{n_k}=x_i^{n_i}x_k^jx_k^{n_k-j}=ex_k^{n_k-j}=e$ and hence $AA=\{e\}$, which contradicts our assumption.
\end{proof}

\begin{lemma}\label{l:bounded-H} Let $X$ be a periodic commutative semigroup with a unique idempotent $e$ and bounded maximal subgroup $H_e$. If $X$ is not bounded, then there exists an infinite subset $A\subseteq X$ such that $AA=\{e\}$.
\end{lemma}

\begin{proof} Since $H_e$ is bounded, there exists a number $p\in\IN$ such that $x^p=e$ for all $x\in H_e$. Assuming that $X$ is not bounded, we conclude that the subsemigroup $P=\{x^p:x\in X\}$ of $X$ is not bounded. We claim that $P\cap H_e=\{e\}$. Indeed, for every $x\in X$ with $x^p\in H_e$, we have $xe\in H_e$ by Lemma~\ref{l:C-ideal} and hence $x^p=x^pe=(xe)^p=e$. Since the maximal subgroup of $P$ is trivial, we can apply Lemma~\ref{l:bounded-e} and find an infinite set $A\subseteq P\subseteq X$ such that $AA=\{e\}$. 
\end{proof}

Our final lemma implies Theorem~\ref{t:main1}.

\begin{lemma}\label{l:main1} For  a unipotent commutative semigroup $X$ the following conditions are equivalent:
\begin{enumerate}
\item $X$ is $\C$-closed;
\item $X$ is periodic, nonsingular and group-bounded;
\item $X$ is bounded and not singular.
\end{enumerate}
\end{lemma}

\begin{proof} The equivalence $(1)\Leftrightarrow(2)$ follows from Theorem~\ref{t:C-closed}, and $(3)\Ra(2)$ is trivial. The implication $(2)\Ra(3)$ follows from Lemma~\ref{l:bounded-H}.
\end{proof}

\section{Proof of Theorem~\ref{t:main2}}

Theorem~\ref{t:main2} follows from Lemma~\ref{l:iC1} and \ref{l:iC-single} below.

\begin{lemma}\label{l:iC1} If a unipotent semigroup $X$ is injectively $\C$-closed, then its center $Z(X)$ is bounded, nonsingular, and group-finite.
\end{lemma} 

\begin{proof} By Theorem~\ref{t:center}, the semigroup $Z(X)$ is periodic, nonsingular and group-finite. If $Z(X)$ is empty, then $Z(X)$ is bounded. If $Z(X)$ is not empty, then by the periodicity, $Z(X)$ contains an idempotent and hence is unipotent, being a subsemigroup of the unipotent semigroup $X$. By Lemma~\ref{l:main1}, $Z(X)$ is bounded.
\end{proof}

\begin{lemma}\label{l:iC-single} Every bounded nonsingular group-finite unipotent commutative subsemigroup $X$ of a $T_1$ topological semigroup $Y$ is closed and discrete in $Y$.
\end{lemma}

\begin{proof}  Replacing $Y$ by the closure of $X$, we can assume that $X$ is dense in $Y$.

\begin{claim}\label{l:Pe-ideal2} For every $x\in X$ and $y\in Y$ there exists a neighborhood $U\subseteq Y$ of $y$ such that the set $x(U\cap X)$ is finite.
\end{claim}

\begin{proof}  To derive a contradiction, assume that there exists $a\in X$ and $y\in Y$ such that for every neighborhood $U\subseteq Y$ of $y$ the set $a(U\cap X)$ is infinite. The periodicity of $X$ ensures that $X=\bigcup_{\ell\in\IN}\korin{\ell}{H_e}$ where $\korin{\ell}{H_e}=\{x\in X:x^\ell\in H_e\}\subseteq \korin{\ell+1}{H_e}$, see Lemma~\ref{l:C-ideal}. This lemma also implies that the set $$\korin{1}{H_e}\cdot Y=H_eY=H_e\overline{X}\subseteq\overline{H_eX}=\overline{H_e}=H_e$$is finite.

Let $k$ be the largest number such that for every $x\in \korin{k}{H_e}$ there exists a neighborhood $U\subseteq Y$ of $y$ such that the set $x(U\cap X)$ is finite.  Since $1\le k<\min\{\ell:a\in \korin{\ell}{H_e}\}$, the number $k$ is well-defined.

\begin{subclaim}\label{sc:improve} For every $x\in \korin{k}{H_e}$ there exists a neighborhood $V\subseteq Y$ of $y$ such that the set $xV$ is a singleton in $X$.
\end{subclaim}

\begin{proof} By the choice of $k$, for every $x\in \korin{k}{H_e}$ there exists a neighborhood $U\subseteq Y$ of $y$ such that the set $x(U\cap X)$ is finite and hence closed in the $T_1$-space $Y$. Then
$$xy\in xU\subseteq x\overline{(U\cap X)} \subseteq \overline{x(U\cap X)}=x(U\cap X)\subseteq X.$$
Since the space $Y$ is $T_1$ there exists an open neighborhood $W$ of $xy$ such that $W\cap x(U\cap X)=\{xy\}$.
 By the continuity of the semigroup operation, the point $y$ has an open neighborhood $V\subseteq U$ such that $xV\subseteq W$.
Then $$xV=x\overline{V\cap X}\subseteq\overline{x(V\cap X)}\subseteq\overline{xV\cap x(U\cap X)}=\overline{\{xy\}}=\{xy\}\subseteq X.$$
\end{proof}

By the maximality of $k$, there exists $b\in \korin{k+1}{H_e}$ such that for every neighborhood $V\subseteq Y$ of $y$ the set $b(V\cap X)$ is infinite. It follows from $b\in \korin{k+1}{H_e}$ that $b^{k+1}\in H_e$ and hence $(b^2)^{k}=b^{2k}=b^{k+1}b^{k-1}\in H_eX^1\subseteq H_e$ and hence $b^2\in \korin{k}{H_e}$. By Subclaim~\ref{sc:improve}, there exists a neighborhood $U\subseteq Y$ of $y$ such that the set $b^2U$ is a singleton in $X$. Choose any $u\in U\cap X$ and applying Lemma~\ref{l:C-ideal}, conclude that $(b^2u)^k=b^{2k}u^k\in H_eX\subseteq H_e$ and $b^2y=b^2u\in \korin{k}{H_e}$. By Subclaim~\ref{sc:improve}, there exists a neighborhood $V\subseteq U$ such that $(b^2y)V$ is a singleton in $X$. Then the set $A=b(V\cap X)$ is infinite but $$AA\subseteq b^2VV\subseteq b^2UV=(b^2y)V$$ is a singleton. But this contradicts the nonsingularity of $X$.
\end{proof}

\begin{claim}\label{l:improve2} For every $x\in X$ and $y\in Y$ there exists a neighborhood $V\subseteq Y$ of $y$ such that $xV$ is a singleton in $X$.
\end{claim}

\begin{proof}
By Claim~\ref{l:Pe-ideal2}, there exists a neighborhood $U\subseteq Y$ of $y$ such that the set $x(U\cap X)$ is finite and hence closed in the $T_1$-space $Y$. Then
$$xy\in xU\subseteq x\overline{(U\cap X)}\subseteq \overline{x(U\cap X)}=x(U\cap X)\subseteq X.$$
Since the space $Y$ is $T_1$, there exists an open neighborhood $W$ of $xy$ such that $W\cap x(U\cap X)=\{xy\}$.
 By the continuity of the semigroup operation, the point $y$ has an open neighborhood $V\subseteq U$ such that $xV\subseteq W$.
Then
$$xV\subseteq x\overline{V\cap X}\subseteq\overline{x(V\cap X)}\subseteq\overline{ xV\cap x(U\cap X)}=\overline{\{xy\}}=\{xy\}\subseteq X.$$
\end{proof}

\begin{claim}\label{l:improve3} For every $n\in\IN$, $x\in X$ and $y\in Y$ there exists a neighborhood $V\subseteq Y$ of $y$ such that $xV^n$ is a singleton in $X$.
\end{claim}

\begin{proof} For $n=1$ the statement follows from Claim~\ref{l:improve2}. Assume that for some $n\in\IN$ we know that
for every $x\in X$ and $y\in Y$ there exists a neighborhood $U\subseteq Y$ of $y$ such that $xU^n$ is a singleton $\{a\}$ in $X$. By Claim~\ref{l:improve2}, there exists a neighborhood $V\subseteq U$ of $y$ such that $aV$ is a singleton in $X$. Then $xV^{n+1}\subseteq xV^nV=aV$ is a singleton in $X$.
\end{proof}

\begin{claim}\label{l:Xk-discrete} For every $k\in\IN$ the subspace $\korin{k}{H_e}\defeq\{x\in X:x^k\in H_e\}$ of $Y$ is discrete.
\end{claim}

\begin{proof} To derive a contradiction, assume that for some $k\in\IN$ the subspace $\korin{k}{H_e}$ is not discrete and let $k$ be the smallest number with this property. Since $\korin{1}{H_e}=H_e$ is finite, $k>1$. Let $y$ be a non-isolated point of $\korin{k}{H_e}$. It follows that $y^{k}\in H_e$ and $(y^2)^{k-1}=y^{k}y^{k-2}\in H_eX^1\subseteq H_e$ and hence  $y^2\in\korin{k-1}{H_e}$. By the minimality of $k$, the space $\korin{k-1}{H_e}$ is discrete. By the continuity of the semigroup operation, there exists a neighborhood $V_0\subseteq Y$ of $y$ such that $V_0V_0\cap \korin{k-1}{H_e}=\{y^2\}$. By Claim~\ref{l:improve2}, we can additionally assume that $V_0y=\{y^2\}$.

By induction we shall construct a sequence of points $(x_n)_{n\in\w}$ in $\korin{k}{H_e}$ and a decreasing sequence  $(V_n)_{n\in\w}$ of open sets in $Y$ such that for every $n\in\w$ the following conditions are satisfied:
\begin{enumerate}
\item[(i)] $x_n\in V_{n}\cap \korin{k}{H_e}\setminus\{x_i\}_{i<n}$;
\item[(ii)] $y\in V_{n+1}\subseteq V_{n}$ and $x_nV_{n+1}=\{y^2\}$.
\end{enumerate}

Assume that for some $n\in\w$ we have chosen a neigborhood $V_n$ of $y$ and a sequence of points $\{x_i\}_{i<n}$. Since $y$ is a non-isolated point of $\korin{k}{H_e}$, there exists a point $x_n$ satisfying the inductive condition (i). Observe that $x_ny\in V_0y=\{y^2\}$. By Claim~\ref{l:improve2}, there exists a neighborhood $V_{n+1}\subseteq V_n$ of $y$ such that  $x_nV_{n+1}=\{x_ny\}=\{y^2\}$. This completes the inductive step.
\smallskip

After completing the inductive construction, we obtain the infinite set $A\defeq\{x_n\}_{n\in\w}\subseteq X_{k}$ such that $AA=\{y^2\}$. But this contradicts the nonsingularity of $X$.
\end{proof}

\begin{claim}\label{l:closed2} For every $k\in\IN$ the set $\korin{k}{H_e}$ is closed in $Y$.
\end{claim}

\begin{proof} To derive a contradiction, assume that for some $k$ the set $\korin{k}{H_e}$ is not closed in $Y$. We can assume that $k$ is the smallest number with this property. Since $\korin{1}{H_e}=H_e$ is finite, $k>1$ and hence $\{x^2:x\in \korin{k}{H_e}\}\subseteq \korin{k-1}{H_e}$.  Fix any point $y\in\overline{\korin{k}{H_e}}\setminus \korin{k}{H_e}$ and observe that
$$y^2\in\overline{\{x^2:x\in \korin{k}{H_e}\}}\subseteq\overline{\korin{k-1}{H_e}}=\korin{k-1}{H_e} \subseteq \korin{k}{H_e},$$see Lemma~\ref{l:C-ideal}. By Claim~\ref{l:Xk-discrete}, the space $\korin{k}{H_e}$ is discrete. Consequently, there exists a neighborhood $U\subseteq Y$ of $y$ such that $UU\cap \korin{k}{H_e}=\{y^2\}$. Since $y\in\overline{\korin{k}{H_e}}\setminus \korin{k}{H_e}$, the set $A=U\cap \korin{k}{H_e}$ is infinite and $AA\subseteq UU\cap \korin{k}{H_e}=\{y^2\}$, which contradicts the nonsingularity of $X$.
\end{proof}

The boundedness of $X$ implies that $X=\korin{k}{H_e}$ for some $k\in \IN$. By Claims~\ref{l:Xk-discrete} and \ref{l:closed2}, the set $\korin{k}{H_e}=X$ is closed and discrete in $Y$.
\end{proof}

\end{document}